\documentclass[10pt]{amsart}
\usepackage{a4,xy,amssymb,epsfig,color,times}
\xyoption{all}

\def\ot{\rightarrow}

\def\T{\operatorname{T}}

\let\cal\mathcal
\def\eps{{\epsilon}}

\def\cO{{\cal O}}

\newcommand{\se}[1]{\begin{equation*}\begin{split}#1\end{split}\end{equation*}}

\newcommand{\C}{\mathbb{C}}
\newcommand{\N}{\mathbb{N}}
\newcommand{\Z}{\mathbb{Z}}
\newcommand{\R}{\mathbb{R}}

\newcommand{\cD}{\mathcal{D}}
\newcommand{\cI}{\mathcal{I}}
\newcommand{\cJ}{\mathcal{J}}
\newcommand{\cC}{\mathcal{C}}

\newcommand{\naive}{\mathsf{naive}}

\newcommand{\vtx}[1]{*+[o][F-]{\scriptscriptstyle #1}}

\newcommand{\mm}{\textrm{min}}
\newcommand{\Tr}{\textrm{Tr}}

\newcommand{\cycle}{\circlearrowright}

\newcommand{\RS}{\ensuremath{\mathsf{RS}}}

\newcommand{\<}{\langle}
\renewcommand{\>}{\rangle}

\newcommand{\End}{\ensuremath{\mathsf{End}}}
\newcommand{\gr}{\ensuremath{\mathsf{gr}}}

\newcommand{\Mod}{\ensuremath{\mathsf{Mod}}}
\newcommand{\TMod}{\ensuremath{\mathsf{Nil}}}
\newcommand{\fg}{\ensuremath{\mathsf{FG}}}
\newcommand{\CS}{\ensuremath{\mathsf{CS}}}

\newcommand{\degree}{\text{deg}}

\newcommand{\Stab}{\ensuremath{\mathsf{Stab}}}
\newcommand{\Ext}{\mathsf{Ext}}
\newcommand{\Dim}{\mathsf{Dim \,}}

\font\ef = eufb10
\newcommand{\ideal}[1]{\ensuremath{\text{\ef{#1}}}}
\newcommand{\grp}[1]{\ensuremath{{\mathsf{#1}}}}

\newtheorem{lemma}{Lemma}[section]

\newtheorem{theorem}[lemma]{Theorem}
\newtheorem{corollary}[lemma]{Corollary}
\newtheorem{property}[lemma]{Property}

\theoremstyle{definition}

\newtheorem{definition}[lemma]{Definition}

{

}

\theoremstyle{remark}

\newtheorem{remark}[lemma]{Remark}

\newcommand{\TC}{\ensuremath{\mathsf{Con}}}

\newcommand{\iss}{\ensuremath{\mathsf{iss}}}

\newcommand{\Mat}{\texttt{Mat}}

\newcommand{\Rep}{\texttt{Rep}}

\newcommand{\ccM}{\texttt{M}}
\newcommand{\Hom}{\textrm{Hom}}

\newcommand{\Image}{\textrm{Im}}
\newcommand{\GL}{\ensuremath{\mathsf{GL}}}
\newcommand{\Glob}{\ensuremath{\mathsf{Glob}}}
\newcommand{\id}{\mathbf{1}}

\title{Noncommutative Tangent Cones and Calabi Yau Algebras}

\author{Raf Bocklandt}
\address{Raf Bocklandt\\University of Antwerp\\Middelheimlaan 1\\B-2020 Antwerpen (Belgium)}
\email{rafael.bocklandt@ua.ac.be}
\thanks{The author is a Postdoctoral Fellow of the Fund for Scientific Research - Flanders (Belgium)}

\xyoption{all}
\begin{document}

\begin{abstract}
We study the generalization of the idea of a local quiver of a representation of a formally smooth algebra, to broader classes
of finitely generated algebras. In this new setting we can construct for every semisimple representation $M$ a local model and a non-commutative tangent cone.
The representation schemes of these new algebras model the local structure and the tangent cone of the representation scheme of the original algebra at $M$. In this way one can try to classify algebras according to their local behavior. As an application we will show that the tangent cones of Calabi Yau $2$ Algebras are always preprojective algebras. For Calabi Yau $3$ Algebras the corresponding statement would be that the local model and the tangent cones derive from superpotentials. Although we do not have a proof in all cases, we will show that this will indeed hold in many cases.
\end{abstract}

\maketitle

\section{Introduction}

Recently, Calabi Yau Algebras were introduced as a noncommutative generalization of Calabi Yau Varieties and they 
form an important tool for physicists in topological string theory. It was noticed that Calabi Yau $3$-Algebras were tightly connected to superpotentials. 

For quotients of path algebras by homogeneous ideals it was proved in \cite{raf} that Calabi Yau $3$-Algebras are indeed algebras with a superpotential and $2$-Calabi Yau Algebras are preprojective algebras (in both cases the converse is not always true). 
There are however lots of algebras that do not fit in this framework (f.i. group algebras, filtered algebras, deformations of graded algebras, etc).
The aim of this paper is to find a way to connect these algebras to the graded case.

The road we will follow is the local study of representations of algebras. We will pursue generalizations
of local structure theorems by Le Bruyn and Procesi \cite{lebruynprocesi} and Crawley-Boevey \cite{CrawleyBoevey}. They proved that
for certain algebras (path algebras for the former, preprojective algebras for the latter), the \'etale local structure of
the representation space around a semisimple representation space can be modeled by a new
algebra which is of the same type (a path algebra or a preprojective algebra).
In this paper we will do a similar thing but instead of looking at the \'etale local structure we will work with formal
completions and tangent cones because this adapts more easily to the noncommutative case.

The commutative picture we want to generalize is the following. Let $X$ be an $n$-dimensional affine variety with coordinate ring $R$. For any point $p\in X$ we look at the corresponding maximal ideal $\ideal m$ and this gives us an $\ideal m$-adic filtration on $R$. Locally we can describe $X$ around $p$ using the formal completion $\hat R_\ideal m$.

We can also consider the associated graded ring $\gr_{\ideal m} R$ and this corresponds geometrically to the tangent cone, and it has the advantage that it is often easier to express in terms of generators and relations. A map $\phi$ between two varieties is \'etale at $p$ if the corresponding map between
the local rings of $p$ and $\phi(p)$ is an isomorphism. This is equivalent to asking that the map between the associated graded rings is an isomorphism
and hence the tangent cone is an important \'etale invariant.

In the setting of finitely generated noncommutative algebras we can define for every semisimple representation a noncommutative tangent cone and a local model in a similar way. Such a tangent cone will always be a graded path algebra of a quiver with relations. We will prove that if the original algebra is Calabi Yau $2$ then these noncommutative tangent cones need not to be Calabi Yau but they will be preprojective algebras. For Calabi Yau 3 algebras similar things happen and in many cases the local model or the noncommutative tangent cone comes from a superpotential. 

The paper is organized in the following way. We start with some preliminaries about graded and complete algebras, path algebras and Calabi Yau algebras.
In section \ref{tangentcone} we introduce the noncommutative tangent cone and in section \ref{repspaces} and \ref{slice} we study the connection with representation schemes and the slice results by Procesi, Le Bruyn and Crawley-Boevey. After that we prove our main result about the tangent cones of Calabi Yau algebras. We end the paper with some examples.

\section{Preliminaries}

\subsection{Graded and Complete Algebras}\label{gradedcomplete}

Let $S$ be a finite dimensional semisimple algebra over $\C$ and $W$ an $S$-bimodule. We define 
two tensor algebras: 
\se{
 &T_SW := S \oplus W \oplus W\otimes_S W \oplus \cdots = \{(a_i)_{i \in \N} | a_i \in W^{\otimes_S i}, a_i=0 \text{ if }i >\!\!>0 \} \\
 &\hat T_SW := S \times W \times W\otimes_S W \times \cdots = \{(a_i)_{i \in \N} | a_i \in W^{\otimes_S i}\} \\
 &(a_i) \cdot (b_i) := (\sum_{0\le k\le i} a_{i-k}\otimes_S b_{k})
 }
Both algebras have a standard projection onto $S$ and we will denote the kernels in both cases by $\ideal W$.
On the former algebra we will put the standard gradation and we will refer to it as the \emph{graded tensor algebra}, while on the latter
we will put the ${\ideal W}$-adic topology and call it the \emph{completed path algebra}.
The degree of an element in the graded algebra will be the maximal degree of its homogeneous components, for the completed algebra
we will define the degree as the minimal degree of its homogeneous components (i.e. $\degree f = \max\{i|f \in \ideal W^i\}$). Furthermore
if $f \in \hat T_S W$ we define $f_\mm \in T_S W$ as the homogeneous component of minimal degree.

To preserve the graded or topological structure of the algebras, the ideals we will consider in the two cases will be graded or closed.
We will call such an ideal $\ideal i$ in \emph{admissible} if it is contained
in $\ideal W^2$. From now on a \emph{graded or complete algebra} $A$ is an algebra of the form 
$T_SW/\ideal i$ or $\hat T_SW/\ideal i$ for some $S,W$ and 
an admissible ideal $\ideal i$. If we write $\ideal i = \<r_i, i \in \cI\>$, we mean that $\ideal i$ is the smallest graded or closed ideal containing
the $r_i$. We will also assume in this notation that the $\{r_i\}$ form a minimal set (i.e. every subset generates a smaller ideal) and
each $r_i$ satisfies $e_1r_ie_2$ for some minimal orthogonal idempotents $e_1,e_2$ in the center of $S$. The first idempotent $h(r_i):=e_1$ will be called the head of $r_i$ and $t(r_i):= e_2$ will be the tail of $r_i$. In the graded case we also suppose that the $r_i$ are homogeneous.

If we want to study the representation theory of graded or completed algebras we can introduce the category of finite dimensional nilpotent modules: $\TMod \hat A_{\ideal m}$. 
In both cases this is the full subcategory of the category of finite dimensional left modules $M$ such that $\rho_M:A \to \End_\C(M)$ factorizes
through some projection $ A \to  A/ \ideal W^i$.  Therefore we can see
$\TMod \hat A_{\ideal m}$ as a direct limit of ordinary finite dimensional module categories:
\[
\TMod \hat A_{\ideal m} \cong \lim_{\rightarrow} \Mod A/\ideal W \hookrightarrow \Mod A/\ideal W^2 \hookrightarrow \Mod A/\ideal W^3 \hookrightarrow \cdots
\]

It is possible that different algebras have the same category of nilpotent modules. We can solve this problem
to restrict to graded or completed \emph{path algebras} with relations.

A \emph{quiver} $Q$ is an oriented graph consisting of a set of vertices $Q_0$, a set of arrows by $Q_1$ and maps $h,t$ that assign to each arrow its head and tail. To a quiver we assign a semisimple algebra $S_Q=\C{Q_0}=\C^{\oplus Q_0}$ and we identify the minimal idempotents with the vertices. We can also construct an $S$-bimodule $W_Q=\C{Q_1}$ with multiplication $v_1av_2=a$ if $v_1=h(a)$ and $v_2=t(a)$ and zero otherwise.
We call $\C Q:=T_{\C Q_0}{\C Q_1}$ the \emph{graded path algebra} and $\C\hat Q:= \hat T_{\C Q_0}\C Q_1$ the \emph{completed path algebra}.
We will also consider graded or completed path algebras with relations which are the corresponding quotients by admissible ideals.
In fact graded or complete path algebras are just graded or complete algebras with a commutative semisimple part. 

We can turn each graded or complete algebra $A$ into a path algebra with relations in by standard Morita equivalence.
Let $S$ be the semisimple part of $A$ and take a minimal idempotent $e_i$ for each simple component of $S$:
$e_iSe_j=\C\delta_{ij}$ and $S e S =S$ with $e=\sum_i e_i$. It is well known $eAe$ is a path algebra with relations and
we have a functor $\TMod \hat A_{\ideal m} \to \TMod eAe: M \to eM$, which is an equivalence of categories.

On the other hand if two completed or graded path algebras with relations have equivalent categories of nilpotent modules
then they are isomorphic.
\begin{lemma}
A completed path algebra with relations $A$ is determined up to isomorphism by the category $\TMod \hat A_{\ideal m}$. 
\end{lemma}
\begin{proof}
We can reconstruct $A$ from $\TMod \hat A_{\ideal m}$ in the following canonical way. 
First define $\Sigma$ to be the semisimple object that is the direct sum of all simple modules with multiplicity one.
We can put a filtration on $\cC_1\subset \cC_2 \subset \cdots \subset \TMod \hat A_{\ideal m}$ by defining $\cC_1$ the subcategory of all semisimples and $M$
in $\cC_i$ if there exists an $N \subset M$ such that $N \in \cC_j$ and $M/N \in \cC_{i-j}$.

Each of these categories is in fact the category of modules
of a finite-dimensional algebra: $\cC_i = \Mod A/\ideal W^i$. The free $A/\ideal W^i$-module of rank one can be categorically defined
as the unique module that is projective in $\cC_i$ and has top $\Sigma$. We will denote this module by $F_i$.

Between the $F_i$ we can chose surjective maps $\dots \to F_2 \stackrel{\pi_1}{\to} F_1 \stackrel{\pi_0}{\to} F_0=\Sigma$.
If $A$ is complete, we can construct $A$ from the opposite algebras of morphisms:
\se{\lim_{\rightarrow}  \Hom(F_i,F_i)^{op} \ot \Hom(F_{i+1},F_{i+1})^{op}
&= \lim_{\rightarrow}  A/\ideal W^i \ot A/\ideal W^{i+1}\\
&= A.}

If $A$ is graded then we have to take the associated graded of this limit construction.
\end{proof} 

We will also need Ext-groups in $\TMod \hat A_{\ideal m}$ and projective resolutions for completed and graded path algebras with relations. The basic projective modules of $A=\C Q/\ideal i$ or $\C \hat Q/\ideal i$ are of the form $P_i= Ai$ where $i$ is a vertex of $Q$.
Every simple module in $\TMod \hat A_{\ideal m}$ has a minimal resolution of projective $A$-modules:
\[
\cdots \to P^3 \to P^2=\oplus_{t(r)=i}P_{h(r)}\stackrel{\cdot ra^{-1}}{\to} P^1=\oplus_{t(a)=i}P_{h(a)}\stackrel{\cdot a}{\to} P_i\to S_i \to 0
\]
and for every map $P^\ell\to P^{\ell-1}$, the image sits inside $\ideal WP^{\ell}$ and we have that
\[
 \Ext^\ell_{\TMod \hat A_{\ideal m}}(S_i,S_j) = \Hom_{A}(P^\ell,S_j).
\]

Finally, the associated graded of a complete path algebra with relations $\hat A$ is by definition $\gr \hat A := \oplus_{0\le i} {\ideal W^i}/{\ideal W^{i+1}}$. If $\hat A = \C \hat Q/\ideal i$ then  $\gr \hat A = \C Q/\gr \ideal i$ with $\gr \ideal i=\<f_\mm| f \in \ideal i\>$. 
Naively, one would think that the generators of this new ideal are the lowest degree parts of the generators of $\ideal i$, however this only holds for special choices of generator sets.
Such a set will be called \emph{gradable}. 
F.i. $\{XY+Z^3,YX+Z^3\}$ is not gradable because $\gr \ideal i = \<XY,YX,XZ^3-Z^3X, YZ^3-Z^3Y\>$, however the set 
$\{XY+XYX,YX+XYX\}$ is gradable because $XYX$ already sits in $\<XY,YX\>$ and hence these commutators as well.

Using this type of reasoning one can see that a set $\{r_i, i \in \cI \} \subset \ideal i$ is gradable if for every
linear combination $\sum_k a_k (r_{i_k})_{\mm} b_k$ that is zero in $\C Q$ we have that $(\sum_k a_kr_{i_k}b_k)_{\mm}$ is zero inside
$A_{\naive} = \C Q/\<r_{i\mm} \in \ideal i\>$. 
For each base element in each $\Ext_{A_{\naive}^3}(S_i,S_j)$ we can find a relation between the relations (the corresponding syzygy) and its suffices
to check it for these relations (as they generate the rest).

\subsection{Calabi Yau Algebras}

\begin{definition}
An algebra $A$ is Calabi Yau \emph{Calabi Yau of dimension $n$} 
if in the derived category of finite dimensional modules $\cD^b \mod A$ 
there are natural isomorphisms 
\[
\nu_{M,N}: \Hom_{\cD^b \Mod A}(M,N) \to \Hom_{\cD^b \Mod A}(M,s^n N)^*, \text{ ($^*$ is the complex dual)}
\]
here $M,N$ are complexes of modules and $s$ is the shift functor of the derived category.
\end{definition}
For further details about this property we refer to \cite{raf}, \cite{ginzburg}
In this paper we will only need the following result:

\begin{property}
If $A$ is Calabi Yau of dimension $n$ then 
\begin{itemize}
\item[C1]
The global dimension of $A$ is also $n$.
\item[C2] 
If $X, Y \in \Rep A$ then
\[
\Ext^k_A(X,Y) \cong \Ext^{n-k}_A(Y,X)^*.
\]
\item[C3]
The identifications above gives us a pairings $\<,\>^k_{XY}: \Ext^k_A(X,Y) \times \Ext^{n-k}_A(Y,X) \to \C$ which satisfy
\[
\<f,g\>^k_{XY} = \<1_X,g*f\>^0_{XX} = (-1)^{k(n-k)}\<1_Y,f*g\>^0_{YY}, 
\]
where $*$ denotes the standard composition of extensions.
\end{itemize}
\end{property}

To state the connection between graded algebras, we need some extra definitions:

\begin{definition}
\begin{itemize}
 \item[]
\item
If $Q$ is a quiver then the double of $Q$ is the quiver $Q^d\supset Q$ that  contains for every arrow $a \in Q_1$
an extra arrow $a^*$ with $h(a^*)=t(a)$ and $t(a^*)=h(a)$. The preprojective algebra of $Q$ is then defined as
\[
\Pi(Q^d) := \C Q^d / \<\sum_{h(a)=i} aa^* - \sum_{t(a)=i} a^*a, i \in Q_0\>
\]
\item
For a given quiver the vector space $\C Q/[\C Q,\C Q]$ has as basis the set of cycles up to cyclic permutation of the arrows. We can embed this space into $\C Q$ by mapping a cycle onto the sum of all its possible cyclic permutations:
\[
\cycle : \C Q/[\C Q,\C Q] \to \C Q: a_1\cdots a_n \mapsto \sum_i a_i\cdots a_na_1\cdots a_{i-1}.
\] 
The elements of this vector space are called superpotentials.

Another convention we will use is the inverse of arrows: if $p:= a_1\cdots a_n$ is a path and $b$ an arrow, then
$pb^{-1}=a_1\cdots a_{n-1}$ if $b=a_n$ and zero otherwise. Similarly one can define $b^{-1}p$. These new defined maps can be combined to obtain a 'derivation' 
\[
\partial_a : \C Q/[\C Q,\C Q] \to \C Q : p \mapsto \cycle(p)a^{-1} =a^{-1}\cycle(p).
\]

An algebra with a superpotential is an algebra of the form
\[
A_W := \C Q/\<\partial_a W, a \in Q_1\>\text{ with }W \in \C Q/[\C Q,\C Q]
\]

\end{itemize}
\end{definition}

\begin{theorem}[\cite{raf}]\label{calab}
Let $A$ be a graded path algebra with relations.
\begin{itemize}
\item
$A$ is Calabi Yau of dimension $2$ if and only if $A$ is the preprojective algebra of a non-Dynkin quiver 
\item 
If $A$ is Calabi Yau of dimension $3$ then $A$ is an algebra with a superpotential, but not all superpotentials
give rise to Calabi Yau Algebras.
\end{itemize}
\end{theorem}

\section{The noncommutative tangent cone}\label{tangentcone}

We suppose that $A=\C\<X_1,\dots X_k\>/\<F_1, \dots, F_l\>$ is a noncommutative finitely generated algebra with a finite number of relations. 
Let $M$ be a finite dimensional semisimple module of $A$ which can be decomposed in simples $M=S_1^{\oplus \eps_1}\oplus\dots \oplus S_k^{\oplus \eps_k}$.
The action of $A$ on $M$ gives us a map $\rho_M: A \to \End M$. The image of this map is 
$\End_\C S_1 \otimes 1_{\eps_1}\oplus \cdots \oplus \End_\C S_k\otimes 1_{\eps_k}$ and this is isomorphic to the semisimple algebra
\[
S = \Mat_{\dim S_1}(\C) \oplus \dots \oplus  \Mat_{\dim S_k}(\C).
\]
We will denote the idempotents in $S$ that correspond to the ones by $e_i$ and the kernels of the maps $A \to e_iS$ by $\ideal s_i$. 
Using this notation we have
the kernel of the map $A \to S$ is $\ideal m := \cap_i \ideal s_i$.

We can consider the $\ideal m$-adic filtration on $A$
\[
F_{i}A = 
\begin{cases}
\ideal m^{-i}& i<0\\
A&i>0.
\end{cases}
\]
This filtration gives rise to two new algebras: the associated graded $\gr_{\ideal m}A$ and the formal completion $\hat A_{\ideal m}$
\se{
\gr A &= \bigoplus_{i=0}^\infty \frac{\ideal m^{i}}{\ideal m^{i+1}}\\
\hat A_{\ideal m} &= \lim_{\leftarrow} (A/\ideal m \leftarrow A/\ideal m^{2} \leftarrow \dots)\\
&= \{(m_i)_{i \in \N}| m_i \in A/\ideal m^i, m^i + \ideal m^j = m^j \text{ if $i>j$}\}
}
There is a natural map $A \to \hat A$ and its kernel is $\cap_i \ideal m^i$.

In both new rings we can identify an ideal that substitutes $\ideal m$: 
\se{
\gr \ideal m &= \oplus_{i\ge 1} \frac{\ideal m^{i}}{\ideal m^{i+1}}\\
\hat {\ideal m} &= \lim_{\leftarrow} (\ideal m/\ideal m^{2} \leftarrow \ideal m/\ideal m^{3} \leftarrow \dots)
}

The first observation we can make is 
\begin{lemma}
$\gr_{\ideal m} A$ is a graded algebra and $\hat A_{\ideal m}$ a complete algebra in the sense
of section \ref{gradedcomplete}.
\end{lemma}
\begin{proof}
First we show that $\hat A_{\ideal m}$ contains $S$. We construct embeddings $S \to A/\ideal m^i$ by induction. For $i=1$ this is trivial
and the embedding $S \to A/\ideal m^i$ can be lifted to $A/\ideal m^{i+1}$ because $S$ is formally smooth \cite{quillen} and $\ideal m^i/\ideal m^{i+1}$ is nilpotent
in $A/\ideal m^{i+1}$. Hence we get a commuting diagram
\[
 \xymatrix{
S\ar@{^(->}[d]\ar@{->}[dr]\ar@{->}[drr]\ar@{->}[drrr]&&&\\
\frac{A}{\ideal m} &\frac{A}{\ideal m^2}\ar[l] &\frac{A}{\ideal m^3}\ar[l] &\dots \ar[l]
}
\]
which we can combine to an embedding $S \to \hat A_{\ideal m}$.
This establishes that $\gr_{\ideal m}A$ and $\hat A_{\ideal m}$ are both $S$-algebras.

We will now chose a set $\{m_\kappa\} \subset A$ such that 
the $\bar m_\kappa = m_\kappa + \ideal m^2$ form a minimal set of generators for $\ideal m/\ideal m^2$ as an $S$-bimodule
and each $\bar m_\kappa$ sits inside a $h(m_\kappa)\ideal m/\ideal m^2t(m_\kappa)$, where $h(m_\kappa),t(m_\kappa) \in \{e_1,\dots ,e_k\}$.
These can be use to construct maps
\se{
 T_S \frac{\ideal m}{\ideal m^2} \to \gr_{\ideal m}A &: \bar m_\kappa \to \bar m_\kappa\\
 T_S \frac{\ideal m}{\ideal m^2} \to \hat A_{\ideal m} &: \bar m_\kappa \to h(m_\kappa)m_\kappa t(m_\kappa).
}
The first map is trivially surjective. The second one is surjective because
all maps $T_S \frac{\ideal m}{\ideal m^2} \to A/\ideal m^i$ are.
\end{proof}
\begin{remark}
In the proof we did not assume that the set $\{m_\kappa\}$ was finite, however we will see further on that this will always be the case.
\end{remark}

\begin{definition}
\begin{itemize}
\item[]
\item 
The \emph{tangent cone}  of $A$ at $M$ is the unique graded path algebra $C_M A$ with relations that is Morita equivalent with $\gr_{ideal m}A$.
\item
The \emph{local model}  of $A$ at $M$ is the unique completed path algebra $L_M A$ with relations that is Morita equivalent with $\hat A_{\ideal m}$.
\item
The \emph{local quiver} of $A$ at $M$ is the quiver $Q_M$ that underlies both $C_M A$ and $T_M A$. Its vertices correspond to the isomorphism classes of simple
factors in $M$ and arrows to the $m_i$. 
\item
The \emph{local dimension vector} of $A$ at $M$ is the map $\alpha_M :(Q_M)_0 \to \N : S_i \to \eps_i$  that assigns to each vertex the multiplicity of the corresponding simple in $M$.
\end{itemize}
\end{definition}
From the construction and the Morita equivalence we can deduce that the associated graded of $L_MA$ at the ideal generated by the arrows is the algebra $C_MA$.

Now let $\Mod A$, $\TMod \gr_{\ideal m}A$ and $\TMod \hat A_{\ideal m}$  denote the categories of (nilpotent) finite dimensional
modules. We can embed $\TMod \hat A_{\ideal m}$ fully and exact in $\Mod A$ because
$\hat A_{\ideal m}/\hat {\ideal m}^i \cong A/\ideal m^i$. 
There is pure categorical description of this embedding.
\begin{lemma}
$\TMod \hat A_{\ideal m}$ is the full subcategory containing all modules whose composition factors are contained in $\{S_1,\dots,S_k\}$.
\end{lemma}
\begin{proof}
First note that submodules, kernels and cokernels of maps in $\TMod \hat A_{\ideal m}$
are also in $\TMod \hat A_{\ideal m}$. In particular this means that the factors of a decomposition series of 
a nilpotent module are also nilpotent.

If a simple module $N$ is nilpotent then $\ideal m^lN=0$ for some $l$. The simplicity of $N$ then implies that $\ideal mN=0$ 
and hence there is an $i$ such that $\ideal s_i N=0$ and hence $N \cong S_i$.
\end{proof}

The embedding $\TMod \hat A_{\ideal m} \hookrightarrow \Mod A$ is exact so we can relate the $\Ext$-groups of both categories:
every element in $\Ext^i_{\hat A_{\ideal m}}(U,V)$ corresponds to an exact sequence $W^\bullet$ in $\TMod \hat A_{\ideal m}$ which is also exact in $\Mod A$ and 
we can map it to the corresponding element in $\Ext^i_{\Mod A}(U,V)$. The mapping is well-defined because if $W^\bullet \to X^\bullet$ is a quasi-isomorphism in $\TMod \hat A_{\ideal m}$ then it also is a quasi-isomorphism in $\Mod A$.
This means we have canonical morphisms $\Ext^i_{\TMod \hat A_{\ideal m}}(U,V) \to \Ext^i_{\Mod A}(U,V)$.
\begin{lemma}\label{exts}
The canonical isomorphism $\Ext^i_{\TMod \hat A_{\ideal m}}(U,V) \to \Ext^i_{\Mod A}(U,V)$ is 
\begin{enumerate}
 \item a bijection if $i=1$.
 \item an injection if $i=2$.
\end{enumerate}
\end{lemma}
\begin{proof}
If $0\to U \to E \to V \to 0$ is a short exact sequence in $\Mod A$ and $U,V \in \TMod \hat A_{\ideal m}$ then $E$ also sits inside
$\TMod \hat A_{\ideal m}$ because its set of composition factors is the union of those of $U$ and $V$. So for $i=1$ the morphism is surjective.
It is also injective because the direct sum of $U$ and $V$ is the same in both categories.

To prove the injectivity if $i=2$, we use a well known criterion to check whether an exact sequence is
trivial: $0 \to A \to B \to C \to D \to 0$ is trivial as an extension in $\Mod A$ if and only if there exists a $J \in \Mod A$ such that
$A \to \Image(A \to B) \oplus J \to B$ is a short exact sequence. But as $A,B \in \TMod \hat A_{\ideal m}$ then the middle term is also nilpotent and hence
$J$ as well, so $0 \to A \to B \to C \to D \to 0$ is trivial in $\TMod \hat A_{\ideal m}$.
\end{proof}
\begin{remark}
We can use the bijectivity of the first Ext-spaces to show that the number of arrows in the quiver is finite. The number of
arrows in $Q_M$ from $j$ to $i$ equals $\Ext_{L_MA}^1(\C i,\C j)$, where $\C i$ is the standard simple $L_MA$-module corresponding to the vertex $i$.
By Morita equivalence this equals to $\Ext_{\hat A}^1(S_i,S_j)$ and by bijectivity this 
equals $\Ext_{A}^1(S_i,S_j)$. This last space is finite dimensional because $A$ is finitely generated.
\end{remark}

In general the maps between the higher Ext-spaces are neither injective nor surjective, but we can obtain some surjectivity by extending $M$ to a
bigger semisimple module.
\begin{lemma}
If $A$ is an ext-finite algebra and $M$ a semisimple module then for a given $k \in \N$, there exists a semisimple $A$-module $N$ such that $M \subset N$ and
the maps $\Ext^l_{\TMod \hat A_{\ideal m}}(S_i,S_j) \to \Ext^l_{\Mod A}(S_i,S_j)$ are all surjective for all $S_i,S_j$ that are simple factors of $M$ and $l\le k$.
\end{lemma}
\begin{proof}
Chose bases for each of the spaces $\Ext^l_{\Mod A}(S_i,S_j)$ and find a representative for each basis element as an exact sequence.
Let $N$ be the direct sum of $M$ and all composition factors of the modules that occur in these exact sequences. 
\end{proof}
Such an $N$ as in the lemma above will be called an \emph{extension of $M$ towards surjectivity}.

\begin{theorem}\label{towardssurjectivity}
Suppose $N$ is a semisimple representation of $A$ and $L_N A = \C \hat Q_N/(r_i, i \in \cI)$. If $M$ is a submodule of $N$ then $Q_M$ will be the quiver
obtained by deleting the vertices of the factors that do not occur in $M$ and all arrows connected to them. There is
a standard projection $\pi: \C \hat Q_N \to \C \hat Q_M$ that identifies all common vertices and arrows and maps all the other vertices and arrows to zero. In this notation we have that
\[
 L_M A = \C Q_M/(\pi(r_i), i \in \cI)
\]
\end{theorem}
\begin{proof}
We have to show that $L_MA = L_NA/\<e_i | S_i \not \subset M\>$. 
As both are completed path algebras with relations we just have to show that their nilpotent module categories are equivalent.

The category of nilpotent $L_NA/\<e_i | S_i \not \subset M\>$-modules is the subcategory of $\TMod L_N A$ that contains
the modules that factorize through  $L_NA/\<e_i | S_i \not \subset M\>$. 
These are the modules that are annihilated by the $e_i$ with $S_i \not \subset M$, i.e. their composition factors
are submodules of  $L_NA/\<L_NA_{\ge 1},(\<e_i | S_i \not \subset M\>L_N A)\>$ but seen as a module
in $\T_N\Mod A \subset \mod A$. This exactly corresponds to the module $\oplus_{S_i \subset M} S_i$.
\end{proof}

\section{The connection with the representation schemes}\label{repspaces}

In this section is $A$ an algebra that is the quotient of a path algebra $\C Q$ with an ideal
$\ideal i=\<r_i | i \in \cI\>$ that is not necessarily graded but we suppose that $\ideal i \cap \C Q_0=0$ and that the $r_i \in e_{h_i}\C Qe_{t_i}$. 

To a dimension vector $\alpha: Q_0 \to  \N$ with $|\alpha|=n$, we can assign a vector space 
\[
\Rep_\alpha Q := \bigoplus_{a \in Q_1} \Mat_{\alpha_{h(a)}\times \alpha_{t(a)}}(\C)
\]
Equipped with the standard $\GL_\alpha:= \prod_{v \in Q_0} \GL_{\alpha_v}(\C)$ action by conjugation.
We will consider this object as a scheme over the complex numbers.
The points of this scheme can be seen as representations of $\C Q$ and the orbits under the action
are the isomorphism classes of representations.

The coordinate ring of this scheme is a polynomial ring that has for every arrow $a$ $\alpha_{h(a)}\alpha_{t(a)}$ variables corresponding to the entries in the matrix that represents it: $\C[\Rep_\alpha Q] = \C[f_a^{ij}| a \in Q_1]$.
If $p=a_1\cdots a_k \in v\C Qw$ we will denote the function it induces on $\Rep_{\alpha} Q$ by
\[
f_p^{ij} := \sum_{i_1,i_2,i_{k-1}} f_{a_1}^{ii_1}f_{a_2}^{i_1i_2}\cdots f_{a_k}^{i_{k-1}j}
\]
If $\ideal m$ is an ideal in $\C Q$ there is a corresponding ideal in $\C[\Rep_{\alpha} Q]$
by 
\[
 \ideal m_{\alpha}=\<f_r^{ij}| r \in v\ideal m, v,w \in Q_0\>
\]

We define the representation scheme $\Rep_\alpha A$ by its ring of functions:
\[
 \C[\Rep_{\alpha} A] = \C[\Rep_{\alpha} Q]/\ideal i_{\alpha}. 
\]
The points of the associated scheme can be seen as representations of $A$ and orbits of the induced action of $\GL_\alpha$ 
are the isomorphism classes.
The action defines a ring of invariants $\C[\iss_{\alpha} Q] := \C[\Rep_{\alpha} Q]^{\GL_{\alpha}}$ and from geometrical representation theory
we know that the embedding $\C[\iss_{\alpha} A] \subset \C[\Rep_{\alpha} A]$ corresponds to 
a quotient map of schemes $\Rep_{\alpha} A \to \iss_{\alpha} A$ which maps every representation to the isomorphism class of its semisimplifcation. 
The main problem in geometric representation theory is to study the geometry of this quotient map.

Another way to look at the ring above is using the $\alpha^{th}$ root of $A$ \cite{booklieven}. This is
the centralizer of $\Mat_n$ in the amalgamated product of $A$ and $\Mat_n$ over $\C Q_0$.
\[
 \sqrt[\alpha]A := (A *_{\C Q_0} \Mat_n)^{\Mat_n}
\]
$\Mat_n$ contains $\C Q_0$ by identifying the each vertex $v$ with a diagonal idempotent of rank $\alpha_v$.
(Reminder: $A *_C B$ is the universal algebra that has embeddings $\iota_A,\iota_B:A,B \hookrightarrow A*_C B$ that agree on the subalgebra $A$, $B$, 
such that any other pair of maps $\psi_A,\psi_B: A,B\to R$ that agree on $C$ factors through $A *_{C} B$.
It consists of linear combinations of products $a_1 * \cdots * a_\ell$ with $a_i \in A \cup B$
subject to the relations that $\cdots * a_j * a_{j+1}* \cdots= \cdots *(a_ja_{j+1} )* \cdots$ if $a_j,a_{j+1}$ are both in $A$ or both in $B$.)
From now on we will delete the subscript $\C Q_0$ in the expressions and assume that all $*$ are implicitly taken over $\C Q_0$.

The fact that the root is a centralizer is reflected in the following identity
\[
 \Mat_n(\sqrt[\alpha]A) = A * \Mat_n(\C).
\]
This can be used to define $\sqrt[\alpha]{\ideal i}$ if $\ideal i$ is an ideal of $A$ by
demanding
\[
 \Mat_n(\sqrt[\alpha]{\ideal i}) = \<\ideal i\>.
\]
where $\<\ideal i\>$ is the $A * \Mat_n(\C)$-ideal generated by the elements of $\ideal i \subset A \subset A * \Mat_n(\C)$. 

Using these notations, it can be checked that
\se{
 \C[\Rep_\alpha Q] \cong \frac{\sqrt[\alpha]{\C Q}}{[\sqrt[\alpha]{\C Q},\sqrt[\alpha]{\C Q}]},
 \C[\Rep_\alpha A] \cong \frac{\sqrt[\alpha]{A}}{[\sqrt[\alpha]{A},\sqrt[\alpha]{A}]},
 \ideal i_{\alpha} \cong \frac{\sqrt[\alpha]{\ideal i}}{[\sqrt[\alpha]{\C Q},\sqrt[\alpha]{\C Q}]}.
}
where we divide out the ideals generated by the commutators. 

The $\GL_{\alpha}$-action can also be defined naturally. $GL_\alpha \subset \Mat_n(\C)$ acts by conjugation on $\Mat_n(\C)$ and it fixes $\C Q_0$. We can extend this action to an action on $A * \Mat_n$ by letting it act trivially on $A$. This action
fixes the centralizer of $\Mat_n(\C)$ so we have a $\GL_{\alpha}$-action on $\sqrt[\alpha]{A}$. Finally commutators are mapped to commutators so
the action factors to an action on $\C[\Rep_{\alpha} Q]$.

This way of looking at things has the advantage that it is easily generalizable to the complete case so if 
$\hat A$ is a complete path algebra with relations we define its representation scheme as the scheme associated to the ring
$\C[\Rep_{\alpha}\hat A] := \sqrt[\alpha]{\hat A}/[\sqrt[\alpha]{\hat A},\sqrt[\alpha]{\hat A}]$. The quotient scheme
is again associated to the subring $\C[\iss_{\alpha}\hat A]$ of $\GL_{\alpha}$-invariant elements.

Let $M$ be a semisimple point in $\Rep_\alpha A$ which corresponds to a maximal ideal $\ccM \lhd \C[\Rep_{\alpha} A]$.
Denote the kernel of $\rho_M: A \to \End M$ again by $\ideal m$ and the image of this map by $S$.
The ideal $\ideal m_{\alpha}\subset \ccM$ is the ideal
whose points corresponds to representations $N \in \Rep_{\alpha} A$ such that $\rho_N$ factorizes through $\rho_M$. 
As $S$ is a semisimple algebra the number of isomorphism classes of representations in $\Rep_\alpha S$ is finite, so the points defined by $\ideal m_{\alpha}$ form a finite set of closed orbits in $\Rep_{\alpha} A$. This means that we can see $\ideal m_{\alpha}$ as the intersection
of ideals $\ideal o_i$ each one corresponding to one orbit $\cO(M_i)$ which we represent by a representation $M_i$.
If we decompose
$M$ as $S_1^{\oplus \eps_1}\oplus \dots \oplus S_1^{\oplus \eps_k}$, the $M_i$
form a maximal set of non-isomorphic modules $M_\eta = S_1^{\oplus \eta_1}\oplus \dots \oplus S_k^{\oplus \eta_k}$ with $\eta \in \N^k$ such that $M_\eta \in \Rep_{\alpha} A$. The particular $\eta$ that corresponds to $M$ itself is clearly the local dimension vector $\alpha_M:= \eps$.

For all these ideals above we can construct the associated graded.
$\gr_{\ccM} \C[Rep_\alpha A]$ will describe the tangent cone of $\Rep_\alpha A$ to $M$ while
$\gr_{\ideal o_i} \C[\Rep_\alpha A]$ describes the tangent cone to the orbit. Geometrically
this last scheme can be seen as a fiber bundle over $\GL_\alpha M_i$. Its fibers are
the unions of the tangent line perpendicular to the orbits:
\[
\TC_{\GL_{\alpha}M_i}\Rep_{\alpha} A = \{p + T_{q}\GL_\alpha M| p \in \TC_q\Rep_{\alpha}A, q \in \GL_{\alpha}M\}.
\]
Finally $\gr_{\ideal m_{\alpha}} \C[\Rep_{\alpha}A]$ is the direct sum of each of 
the $\gr_{\ideal o_i} \C[\Rep_\alpha A]$ and hence is corresponding scheme is the union of these 
tangent cones.

Now, we can make the connection with the representations of $\gr_{\ideal m}A$ and $\hat A_{\ideal m}$:
\begin{lemma}
If $A$ is an algebra over $\C Q$, $M$ a semisimple representation with corresponding ideals $\ideal m \lhd A$ and $\ideal m_\alpha \lhd \C[\Rep_\alpha A]$ then
\begin{itemize}
\item $\gr_{\ideal m_\alpha}\C[ \Rep_\alpha A] \cong \C[\Rep_\alpha \gr_{\ideal m} A],$
\item $\widehat{\C[ \Rep_\alpha A]}_{\ideal m_\alpha} \cong \C[\Rep_\alpha \hat A_{\ideal m}],$
\end{itemize}
\end{lemma}
\begin{proof}
We only prove the first statement, the proof of the second is very analogous.
The first step in the proof is to show that taking the associated graded or the completion commutes with the free product
\se{
 (\gr_{\ideal m} A) *_{\C Q_0} \Mat_n(\C) &\cong \gr_{\<\ideal m\>} (A *_{\C Q_0} \Mat_n(\C)).
}
We note that for every $i \in \N$ we have that 
$\< \ideal m\>^i$ is the linear span of products $c_1 * f_1 \cdots c_l * f_l$ with $c_i \in \Mat_n(\C)$ and
$f_j \in \ideal m^{\kappa_j}$ such that $|\vec \kappa|:=\sum_j \kappa_j \ge i$. The rewriting rules for the free product are such that
if we reduce the length of the product by removing a $c_j$ in $\C Q_0 \subset \Mat_n(\C)$ then $f_{j-1}c_jf_{j} \in \ideal m^{\kappa_{j-1}\kappa_{j}}$
so the sum of the $\kappa'$s does not change.
This fact implies $\< \ideal m\>^i \cap A= \ideal m^i$ and  $\< \ideal m\>^i \cap \Mat_n(\C) = 0$ and hence we have injections
\se{
\gr_{\ideal m} A \to \gr_{\<\ideal m \>}(A*\Mat_n(\C))\text{ and } \Mat_n(\C) \to \gr_{\<\ideal m \>}(A*\Mat_n(\C)).
}
By the universal property of the amalgamated product, these maps combine to a big map $\pi: \gr_{\ideal m} A * \Mat_n(\C) \to 
\gr_{\<\ideal m \>}(A*\Mat_n(\C))$ which is a map of graded algebras if we give $\gr_{\ideal m} A * \Mat_n(\C)$ the grading coming
from the grading of $\gr_{\ideal m} A$ and  consider $\Mat_n(\C)$ of degree zero.

We will show that this map is an isomorphism. The surjectivity of $\pi$ follows from the fact that
\[
 \pi( c_1*(f_1 + \ideal m^{\kappa_1})\cdots c_l*(f_l + \ideal m^{\kappa_l})) = c_1*f_1\cdots c_l*f_l + \<\ideal m\>^{\sum_i |\kappa_i|}
\]
For the injectivity we construct a right inverse for every homogeneous part.
So fix a $k \in \N$ and define
\se{
 \iota_k : &\gr_{\<\ideal m \>}(A*\Mat_n(\C))_k \to \gr_{\ideal m} A * \Mat_n(\C):\\ 
&c_0 *f_0 \dots c_l *f_l + \<\ideal m\>^{k+1} \mapsto
\sum_{|\vec \kappa|=k} c_1*(f_1 +\ideal m^{\kappa_1+1})\cdots c_n*(f_l + \ideal m^{\kappa_l+1})
}
The sum is taken over every vector $\vec \kappa \in \N^l$ such that the sum of its coefficients is $k$. Although this sum seems to consist of several terms
it has at most one nonzero term and this will be one coming from the vector $\vec \kappa$ we defined earlier.
The sum seems to depend on the length of the sequence (i.e. $l$) but if we can reduce the length because one of the $c_j$ or $f_j$ is in $\C Q_0$, it
is easy to check that the new sum will give the same answer (in case $f_j$ is a scalar this because $\kappa_j$ then must be zero, in the case
$c_j$ is a scalar we can use the product rule in $\gr_{\ideal m} A$).
Finally one can see that $\iota_k\pi|_{(\gr_{\ideal m} A * \Mat_n(\C))_k} = \id$.

The second step is to restrict to the $\alpha^{th}$ root. If $R$ is a ring and $\ideal r$ an ideal 
then 
\[
\gr_{\Mat_n(\ideal r)}\Mat_n(R) = \Mat_n(\gr_{\ideal r} R).
\]
This follows from the fact that $\Mat_n(\ideal r)^k = \Mat_n(\ideal r^k)$.
If we apply this to the case with $R= \sqrt[\alpha]{A}$ we get
\se{
(\gr_{\<\ideal m\>} A*\Mat_n(\C))^{\Mat_n(\C)} &= 
(\gr_{\Mat_n(\sqrt[\alpha]{\ideal m})}\Mat_n(\sqrt[\alpha]{A}))^{\Mat_n(\C)} \\
&= \Mat_n(\gr_{\sqrt[\alpha]{\ideal m}} \sqrt[\alpha]{A})^{\Mat_n(\C)} =.\gr_{\sqrt[\alpha]{\ideal m}} \sqrt[\alpha]{A}
}

The last step is the compatibility with commutators.
Let $R$ be a ring and $[R,R]$ the ideal generated by the commutators and $\ideal r$ an ideal of $R$.
We have to prove that
\[
 \frac{\gr_{\ideal r} R}{[\gr_{\ideal r} R,\gr_{\ideal r} R]} \cong \gr_{\frac{\ideal r}{[R,R]}}\frac{R}{[R,R]}.
\]
This can be done by a straightforward computation:
\se{
 (\frac{\gr_{\ideal r} R}{[\gr_{\ideal r} R,\gr_{\ideal r} R]})_k 
&=\frac{\ideal r^k/\ideal r^{k+1}}{([R,R]\cap \ideal r^k)/\ideal r^{k+1}}\\
&=\frac{\ideal r^k}{([R,R]\cap \ideal r^k)+\ideal r^{k+1}}\\
&=\frac{\ideal r^k/([R,R]\cap \ideal r^k)}{\ideal r^{k+1}/([R,R] \cap \ideal r^{k+1})}\\
&=(\gr_{\frac{\ideal r}{[R,R]}}\frac{R}{[R,R]})_k.
}
\end{proof}

Finally, to make the connection with the noncommutative tangent cone and the local model, we have to study what happens to the Morita equivalence.

\begin{lemma}\label{fibered}
\begin{itemize}
 \item[] 
\item $\Rep_{\alpha} \gr_{\ideal m}A = \bigsqcup_{\eta, M_\eta \in \Rep_{\alpha} A} \Rep_\eta C_MA \times_{\GL_\eta} \GL_{\alpha}$
\item $\Rep_{\alpha} \hat A_{\ideal m} = \bigsqcup_{\eta, M_\eta \in \Rep_{\alpha} A} \Rep_\eta L_MA \times_{\GL_\eta} \GL_{\alpha}$
\end{itemize}
\end{lemma}
\begin{proof}
Again we only prove the first statement.
Let $e= \sum e_i$ be the idempotent in $A/\ideal m$ such that $e\gr_{\ideal m}A e=C_M A$.
Each of the  $M_\eta \in \Rep_{\alpha} A$ corresponds to the zero representation
of $C_MA$ with dimension vector $\eta$.

The injection $A/\ideal m\subset \gr_{\ideal m}A$ gives us a $\GL_{\alpha}$-equivariant projection $\pi :\Rep_{\alpha} \gr_{\ideal m}A \to \Rep_{\alpha} S$.
Let $X_\eta$ be the fiber of the image of $M_\eta$ under $\pi$.
We have that
\[
 \Rep_{\alpha}\gr_{\ideal m}A = \bigsqcup_{\eta, \sum \eta_j\alpha^j=\alpha} X_\eta \times_{\GL_\eta} \GL_{\alpha}
\]
This is clear from the fact that we can transform every representation of $Rep_{\alpha} A/\ideal m$ to one of the form $\pi(M_\eta)$.
The equivariantness of $\pi$ implies that every representation of $\Rep_{\alpha}\gr_{\ideal m}A$  can be transformed into one of the fibers $X_\eta$, so
each component becomes of the form $X_\eta \times_{\Stab M_\eta} \GL_{\alpha}$ and it is easy to check that $\Stab M_\eta=\GL_\eta$.

So it only remains to prove that $\Rep_{\eta} C_MA \cong X_{\eta}$. Without loss of generality we can assume that $\rho_{M_\eta}(e)$ is diagonal.
This means that every representation $\rho \in X_\eta$ contains certain rows and columns that are zero and if we strip those and restrict to $C_MA = e\gr_{\ideal m}Ae \subset e\gr_{\ideal m}Ae$ we get a representation of $C_MA$ with dimension vector $\eta$.

To show that the map $X_\eta \to \Rep_{\eta} C_MA$ has an inverse we reinsert the zero rows and columns and induce a $\gr_{\ideal m}A$-representation. We get a map $X_\eta \to \Rep_{\eta} C_MA: \rho \to \bar\rho$ with $\forall x \in C_MA: \forall s_1,s_2 \in A/\ideal m:\bar\rho(s_1xs_2)=\rho_{M_\eta}(s_1)\tilde\rho(x)\rho_{M_\eta}(s_2)$ , where the tilde
means that we reinsert the zeros at the proper places. $\bar \rho$ is uniquely defined because $A/\ideal m$ and $\C_MA$ generate $\gr_{\ideal m} A$ as an algebra.
\end{proof}

The two lemmas summarize to the following connection between noncommutative tangent cones and representation spaces.
\begin{theorem}
\begin{itemize}
\item[]
\item 
$\Rep_{\alpha_M} C_MA$ describes the fiber of the tangent cone to the orbit of $M$ in $\Rep_\alpha A$.
$\iss_{\alpha_M} C_MA$ describes the tangent cone to the image of $M$ in $\iss_\alpha A$.
\item
$\Rep_{\alpha_M} L_MA$ describes the fiber of the completization at the orbit of $M$ in $\Rep_\alpha A$.
$\iss_{\alpha_M} L_MA$ describes the completization at the image of $M$ in $\iss_\alpha A$.
\end{itemize}
\end{theorem}
Although $L_MA$ gives more information about the local structure because there might be different complete algebras with the same
associated graded, it is in many cases more usefull to use $C_MA$. $C_MA$ is more easily describable in the path algebra formalism because
its relations are homogeneous, while the relations of $L_MA$ might not even be finite combinations of monomials. Moreover the properties we are interested in can be determined using the tangent cones, as is seen in the following corollaries.
\begin{corollary}
For every $M$ we denote the $\eta$ such that $M_\eta=M$ by $\alpha_M$ and we call it the local dimension vector of $M$.
The couple $(C_MA,\alpha_M)$ is called the tangent cone setting.
\begin{enumerate}
\item 
The dimension of $\iss_\alpha A$ around $M$ is the same as the dimension of $\iss_{\alpha_M} C_M A$.
\item
If $\iss_{\alpha_M} C_MA$ has only one irreducible component then $M$ is also contained in one irreducible component.
\item 
$\iss_\alpha A$ is smooth at $M$ if and only if $\iss_{\alpha_M} C_M A$.
\item
If $M$ is contained in a unique irreducible component then this component contains simple representations if and only if $\iss_{\alpha}C_MA$
contains simple representations.
\end{enumerate}
\end{corollary}
\begin{proof}
The first three statements are easy consequences of the definition of tangent cones. The last one needs a bit more elaboration.
We will prove that the dimension of the generic stabilizer around $M$ in $\Rep_\alpha A$ and around the zero representation in $\Rep_{\alpha_M} C_MA$ 
are the same using the standard formula for the dimension of the generic stabilizer.
\se{
\dim \text{gen Stab}A &=  \dim \GL_{\alpha} - \dim_M \Rep_\alpha A + \dim_M \iss_\alpha A  \\
&=  \dim \GL_{\alpha} -(\dim_0 \Rep_{\alpha_M} C_MA + \dim \GL_{\alpha} -\dim \GL_{\alpha_M}) + \dim_M \iss_{\alpha_M}C_MA  \\
&=  \dim \GL_{\alpha_M} - \dim_0 \Rep_{\alpha_M} C_MA + \dim_M \iss_{\alpha_M} C_MA  \\
&=\dim \text{gen Stab}C_MA
}
In this calculation we substituted $\dim_M \Rep_\alpha A$ by  $\dim_0 \Rep_{\alpha_M} C_MA + \dim GL_{\alpha} -\dim GL_{\alpha_M}$, this follows
from the fact that by theorem \ref{fibered} around $M$ $\Rep_{\alpha} A$ is a fibered product 
\end{proof}

We end this section with a schematic picture of the situation we described:
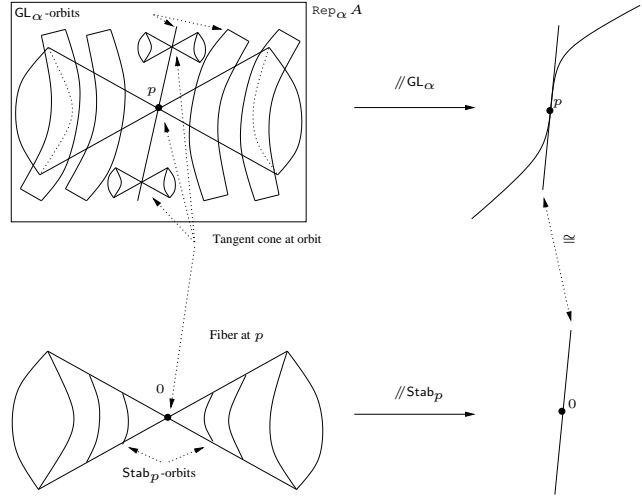
\begin{center}
\begin{picture}(0,0)%
\includegraphics{figcone.pstex}%
\end{picture}%
\setlength{\unitlength}{1657sp}%
\begingroup\makeatletter\ifx\SetFigFontNFSS\undefined%
\gdef\SetFigFontNFSS#1#2#3#4#5{%
  \reset@font\fontsize{#1}{#2pt}%
  \fontfamily{#3}\fontseries{#4}\fontshape{#5}%
  \selectfont}%
\fi\endgroup%
\begin{picture}(9429,7392)(484,-6901)
\put(3466,-4561){\makebox(0,0)[lb]{\smash{{\SetFigFontNFSS{5}{6.0}{\familydefault}{\mddefault}{\updefault}{\color[rgb]{0,0,0}Fiber at $p$}%
}}}}
\put(8731,-3121){\makebox(0,0)[lb]{\smash{{\SetFigFontNFSS{5}{6.0}{\rmdefault}{\mddefault}{\updefault}{\color[rgb]{0,0,0}$\cong$}%
}}}}
\put(2656,-5371){\makebox(0,0)[lb]{\smash{{\SetFigFontNFSS{5}{6.0}{\rmdefault}{\mddefault}{\updefault}{\color[rgb]{0,0,0}$0$}%
}}}}
\put(2521,-871){\makebox(0,0)[lb]{\smash{{\SetFigFontNFSS{5}{6.0}{\rmdefault}{\mddefault}{\updefault}{\color[rgb]{0,0,0}$p$}%
}}}}
\put(8596,-1051){\makebox(0,0)[lb]{\smash{{\SetFigFontNFSS{5}{6.0}{\rmdefault}{\mddefault}{\updefault}{\color[rgb]{0,0,0}$p$}%
}}}}
\put(8821,-5596){\makebox(0,0)[lb]{\smash{{\SetFigFontNFSS{5}{6.0}{\rmdefault}{\mddefault}{\updefault}{\color[rgb]{0,0,0}$0$}%
}}}}
\put(6211,-781){\makebox(0,0)[lb]{\smash{{\SetFigFontNFSS{5}{6.0}{\rmdefault}{\mddefault}{\updefault}{\color[rgb]{0,0,0}$/\!\!/\GL_{\alpha}$}%
}}}}
\put(4996,299){\makebox(0,0)[lb]{\smash{{\SetFigFontNFSS{5}{6.0}{\rmdefault}{\mddefault}{\updefault}{\color[rgb]{0,0,0}$\Rep_{\alpha}A$}%
}}}}
\put(6211,-5416){\makebox(0,0)[lb]{\smash{{\SetFigFontNFSS{5}{6.0}{\rmdefault}{\mddefault}{\updefault}{\color[rgb]{0,0,0}$/\!\!/\Stab_p$}%
}}}}
\put(2161,-6631){\makebox(0,0)[lb]{\smash{{\SetFigFontNFSS{5}{6.0}{\rmdefault}{\mddefault}{\updefault}{\color[rgb]{0,0,0}$\Stab_p$-orbits}%
}}}}
\put(541,254){\makebox(0,0)[lb]{\smash{{\SetFigFontNFSS{5}{6.0}{\rmdefault}{\mddefault}{\updefault}{\color[rgb]{0,0,0}$\GL_{\alpha}$-orbits}%
}}}}
\end{picture}%

\end{center}
It is important to note that although the picture suggests it, the tangent cone bundle does not embed as a scheme inside $\Rep_{\alpha}A$ and in general there need not to be a morphism of schemes between the bundle and the representation scheme. 

\section{Tangent Cones and Slice results}\label{slice}

For special types of algebras there are already results that relate the \'etale local structure
of the representation spaces with path algebras for quivers. We describe the two main examples below.

For the first we recall that a formally smooth algebra is an algebra $A$ that satisfies the following lifting property:
If $\phi; A \to B/\ideal i$ is an algebra morphism and $\ideal i$ is a nilpotent ideal, then we call lift $\phi$ to
a map $\tilde \phi:A \to B$.
\begin{theorem}[Le Bruyn]
If $A$ is a finitely generated formally smooth algebra an $M=S_1^{\oplus \eps_1} \oplus  \cdots \oplus S_k^{\oplus \eps_k}$ a semisimple representation
then there exists a quiver $Q_M$ and a dimension vector $\alpha_M$ such that 
There is a $\GL_{\alpha}$-equivariant morphism 
\[
\phi: \Rep_{\alpha_M} Q \times_{\GL_{\alpha_M}} \GL_{\alpha} \to  \Rep_{\alpha} A
\]
this morphism is \'etale at the point $(0,1)$ with $\phi(0,1)=M$.
\end{theorem}

In \cite{CrawleyBoevey} Crawley-Boevey proved the following result:
\begin{theorem}[Crawley-Boevey]\label{etalprep}
If $M \in \Rep_{\alpha} \Pi(Q^d)$ be a semisimple representation with decomposition
$S_1^{\oplus \eps_1} \oplus  \cdots \oplus S_k^{\oplus \eps_k}$ then there exists a new preprojective algebra $\Pi_M$ such that
there is a $\GL_{\alpha}$-equivariant morphism 
\[
\phi: \Rep_{\alpha_M} \Pi' \times_{\GL_{\alpha_M}} \GL_{\alpha} \to  \Rep_{\alpha} \Pi(Q)
\]
this morphism is \'etale at the point $(0,1)$ with $\phi(0,1)=M$.
The double quiver $Q_M$ underlying $\Pi_M$ has a vertex for each simple component, and the number of arrows from the $i^{th}$ to the $j^{th}$ vertex is
\[
\begin{cases}
2+ \sum_{a \in Q^d_1} \alpha^i(h(a)) \alpha^i(t(a)) - 2\sum_{v \in Q^d_0}\alpha^i(v) \alpha^i(v)  &i=j\\
\sum_{a \in Q^d_1} \alpha^i(h(a)) \alpha^i(t(a)) - 2\sum_{v \in Q^d_0}\alpha^i(v) \alpha^i(v) &i\ne j
\end{cases}
\]
\end{theorem}

In both cases we have a picture like this:
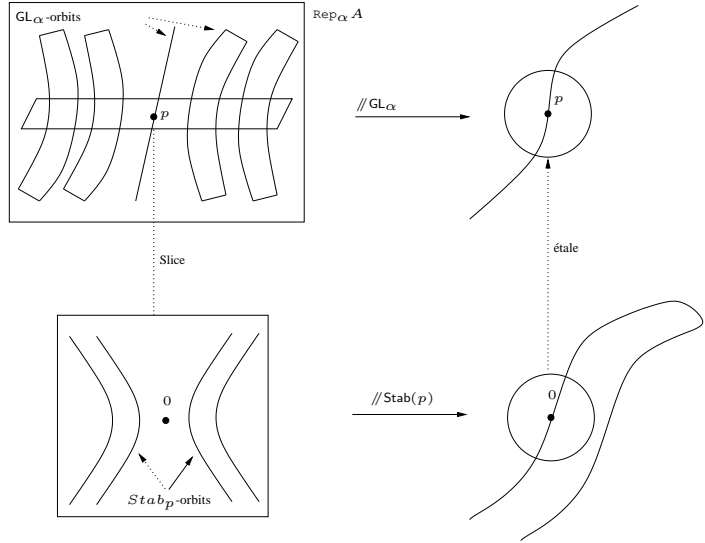
\begin{center}
\begin{picture}(0,0)%
\includegraphics{figet.pstex}%
\end{picture}%
\setlength{\unitlength}{1657sp}%
\begingroup\makeatletter\ifx\SetFigFontNFSS\undefined%
\gdef\SetFigFontNFSS#1#2#3#4#5{%
  \reset@font\fontsize{#1}{#2pt}%
  \fontfamily{#3}\fontseries{#4}\fontshape{#5}%
  \selectfont}%
\fi\endgroup%
\begin{picture}(10394,8079)(484,-7588)
\put(2746,-1231){\makebox(0,0)[lb]{\smash{{\SetFigFontNFSS{5}{6.0}{\rmdefault}{\mddefault}{\updefault}{\color[rgb]{0,0,0}$p$}%
}}}}
\put(8551,-5461){\makebox(0,0)[lb]{\smash{{\SetFigFontNFSS{5}{6.0}{\rmdefault}{\mddefault}{\updefault}{\color[rgb]{0,0,0}$0$}%
}}}}
\put(8641,-3256){\makebox(0,0)[lb]{\smash{{\SetFigFontNFSS{5}{6.0}{\rmdefault}{\mddefault}{\updefault}{\color[rgb]{0,0,0}\'etale }%
}}}}
\put(8641,-1006){\makebox(0,0)[lb]{\smash{{\SetFigFontNFSS{5}{6.0}{\rmdefault}{\mddefault}{\updefault}{\color[rgb]{0,0,0}$p$}%
}}}}
\put(5041,254){\makebox(0,0)[lb]{\smash{{\SetFigFontNFSS{5}{6.0}{\rmdefault}{\mddefault}{\updefault}{\color[rgb]{0,0,0}$\Rep_{\alpha}A$}%
}}}}
\put(2791,-5551){\makebox(0,0)[lb]{\smash{{\SetFigFontNFSS{5}{6.0}{\rmdefault}{\mddefault}{\updefault}{\color[rgb]{0,0,0}$0$}%
}}}}
\put(5671,-1096){\makebox(0,0)[lb]{\smash{{\SetFigFontNFSS{5}{6.0}{\rmdefault}{\mddefault}{\updefault}{\color[rgb]{0,0,0}$/\!\!/\GL_{\alpha}$}%
}}}}
\put(5896,-5506){\makebox(0,0)[lb]{\smash{{\SetFigFontNFSS{5}{6.0}{\rmdefault}{\mddefault}{\updefault}{\color[rgb]{0,0,0}$/\!\!/\Stab(p)$}%
}}}}
\put(586,209){\makebox(0,0)[lb]{\smash{{\SetFigFontNFSS{5}{6.0}{\rmdefault}{\mddefault}{\updefault}{\color[rgb]{0,0,0}$\GL_{\alpha}$-orbits}%
}}}}
\put(2251,-7036){\makebox(0,0)[lb]{\smash{{\SetFigFontNFSS{5}{6.0}{\rmdefault}{\mddefault}{\updefault}{\color[rgb]{0,0,0}$Stab_p$-orbits}%
}}}}
\end{picture}%

\end{center}
Contrarily to the tangent cone the slice maps into $\Rep_{\alpha} A$ and as such it can be considered as a subvariety (modulo the technicalities from the \'etale topology). Also the map from the quotient of the slice to the quotient of the representation scheme is an \'etale map so it might be finite to one instead of one to one. This is illustrated by the second sheet in the quotient of the slice.

To work out the connection between the slice results and the noncommutative tangent cones, we need a reconstruction theorem
\begin{theorem}\label{recon}
If $B$ is a graded path algebra with relations and $\alpha$ is a dimension vector which is nowhere zero, then we can reconstruct 
$B$ from the graded rings $\C[\Rep_{k\alpha} B]$ and the maps
\[
\oplus : \Rep_{l\alpha} B \times \Rep_{k\alpha} B \to \Rep_{(k+l)\alpha} B 
\]
\end{theorem}
\begin{proof}
First of all we note that for any graded path algebra $B$ the variety $\Rep_{k\alpha} B\times_{\GL_{k\alpha}}\GL_{k|\alpha|}$ embeds
as an open and closed subset of $\Rep_n B$ with $n = k|\alpha|$. These embeddings are also compatible with taking the direct sum.

A collection $(f_k)$ of $\GL_{k|\alpha|}$-equivariant functions $f_k: \Rep_{k\alpha}\times_{\GL_{k\alpha}}\GL_{k|\alpha|} B \to \Mat_{k|\alpha|}$ is called a representable sequence if they are homogeneous of the same degree ($\exists \degree f \in \N: f_k(\lambda p)=\lambda^{\degree f} f_k(p)$) and
\[
\forall W_1 \in \Rep_{k\alpha}:\forall W_2 \in \Rep_{l\alpha} : f_{k+l}(W_1\oplus W_2) = f_k(W_1) \oplus f_l(W_2)~~(*).
\]
These sequences generate a ring which we will denote by $\RS(B)$. We will prove that $\RS(B) \cong B$.

From the reconstruction theorem of Procesi \cite{procesi} we know that every
$\GL_n$-equivariant function from $\Rep_n B$ to $\Mat_n(\C)$ is generated by functions of the form $u_b : M \mapsto \rho_M(b)$ and
$t_b : M \mapsto 1_n\Tr \rho_M(B)$. So if $(f)$ is a representable sequence then $f_k$ can be expressed as a noncommutative function
of $F(t_{b_1},u_{b_1},\dots,t_{b_r},u_{b_r})$.
We can use this expression to compute $f_l$ with $l|k$ in different ways. We can use the diagonal embedding of
$\Rep_{l\alpha} B \to \Rep_{k\alpha} B: W \mapsto W^{\oplus k/l}$; this will give us the expression $F(\frac kl t_{b_1},u_{b_1},\dots,\frac kl t_{b_r},u_{b_r})$.
Because $\C[\Rep_{\alpha} B]$ is graded, we can define the zero representation $0$ as the one corresponding to the maximal ideal of positive degree.
The embedding $\Rep_{l\alpha} B \to \Rep_{k\alpha} B: W \mapsto W\oplus 0^{\oplus k-l}$ provides us with another expression for $f_l$:
$F(0,u_{b_1},\dots,0,u_{b_r})$. 

Taking the difference of the two expressions, we obtain a matrix of equations for $k/l$, each at most of degree $\degree f$.
If we chose $k=\degree f!$ this equation has more solutions than its degree so we know that it is zero.
This implies that $F$ does not depend on functions of the form $t_{\dots}$. But $F(0,u_{b_1},\dots,0, u_{b_r}) = u_{F(0,u_{b_1},\dots,0,u_{b_r})}$, so
if $(f)$ is a representable function we can find a sequence $(b_k) \in B^{\N}$ such for $\forall k \in \N:f_k = u_{b_k}$. 

For a given $d= \degree f$ we could find an $\ell$ such that $B/B_{>d}$ is a submodule of a representation in $\Rep_{\ell\alpha}B$. Thus if $k \ge l$
we have that all $u_b, b \in B_d$ are different so $(b_k)$ is a sequence that is constant for large $k$. The diagonal embeddings $\Rep_{l\alpha}\to \Rep_{k\alpha}$  imply that we can even find a sequence $(b_k)$ that is completely constant. So map $B \to \RS(B): b \to (u_b)$ is an isomorphism. 
\end{proof}

\begin{theorem}
Suppose $A$ is a finitely generated algebra and $M$ an $n$-dimensional representation. If $B$ is a graded path algebra with relations such that 
for all $M^k$ there is a $\GL_{k\alpha}$-equivariant morphism 
\[
\phi_k: \Rep_{\alpha_{M^k}} B \times_{\GL_{k\alpha_M}} \GL_{k\alpha} \to  \Rep_{k\alpha} A
\]
which is \'etale at the point $(0,1)$ with $\phi(0,1)=M^k$ and compatible with direct sums i.e. $\oplus_A\circ(\phi_k,\phi_l)=\phi_{k+l}\circ \oplus_B$
with
\se{
\oplus_A &: \Rep_{k\alpha} A \times \Rep_{l\alpha} A \to \Rep_{k+l\alpha} A : (a_1,a_2) \to a_1 \oplus a_2 \\
\oplus_B &: \left(\Rep_{k\alpha_M} B \times_{\GL_{k\alpha_M}} \GL_{k\alpha}\right)\times \left(\Rep_{l\alpha_M} B \times_{\GL_{l\alpha_M}} \GL_{l\alpha}\right)\\ &\to \left(\Rep_{(k+l)\alpha_M} B \times_{\GL_{(k+l)\alpha_M}} \GL_{(k+l)\alpha}\right): (b_1,g_1)\times (b_2,g_2) \mapsto (b_1\oplus b_2, g_1\oplus g_2),
}
then $B \cong L_M A$.
\end{theorem}
\begin{proof}
Because the map is \'etale we have that
\se{
\TC_{\cO(0,1)} \Rep_{\alpha_{M^k}} B \times_{\GL_{k\alpha_M}} \GL_{k\alpha} &\cong  \TC_{\cO(M^k)}\Rep_{k\alpha} A\\
\TC_0\Rep_{\alpha_{M^k}} B \times_{\GL_{k\alpha_M}} \GL_{k\alpha} &\cong  \Rep_{k\alpha_M} C_MA\times_{\GL_{k\alpha_M}} \GL_{k\alpha}\\
\text{so }\TC_0\Rep_{\alpha_{M^k}} B \cong \Rep_{\alpha_{M^k}} B  &\cong  \Rep_{\alpha_{M^k}} C_M A
}
It is clear from the construction that this map is compatible with the graded structure on the rings.
Finally we have to check that there is a compatibility with the direct sum maps, but this follows directly from taking
the associated graded maps of the direct sums $\oplus_A$ and $\oplus_B$.
\end{proof}

The existence of such an algebra $B$ satisfying the condition of this theorem for general $A$ and $M$ is not clear. Even in the commutative case it can go wrong because not every singularity has an \'etale map to it coming from its tangent cone. 
The theorem \ref{etalprep} is hence a stronger than the corresponding theorem \ref{calablocal} but the latter applies in many more cases than the former and can be seen as a usefull generalization. 

It is of crucial importance that the \'etale morphism exists for all the $M^{\oplus k}$ to have the isomorphism.
In many cases there is an isomorphism just for $M$, f.i. $\C\<X,Y\>$ and $\C[X,Y]$ have even isomorphic representation spaces if $n=1$, but
their non-commutative tangent cones are non-isomorphic (in every point $W \in \Rep_1 \C[X,Y]$, $C_W \C[X,Y]=\C[X,Y]$ and $C_W \C\<X,Y\>=\C\<X,Y\>)$.

\section{Tangent Cones of Calabi Yau Algebras}

In this section we will prove that the tangent cones of CY2-algebras are preprojective algebras. We will also discuss
some of the extra troubles that occur in the 3-dimensional case.

\begin{theorem}\label{calablocal}
If $A$ is $2$-Calabi Yau and $M$ is a semisimple $A$-module then $C_M A$ is a preprojective algebra.
\end{theorem}
\begin{proof}
We can extend $M$ towards surjectivity for $2$-extensions by lemma \ref{towardssurjectivity}. Now look at the filter category $\TMod N$
of this new semisimple module $N$. Let us denote the simple factors of $N$ by $S_i, i \in \cI_N$ and the subset of indices that refer to
factors of $M$ by  $I_M$. 
For the $S_i,S_j$ with  $i,j \in I_M$, the $2$-extension spaces in $\T_N\Mod A=\TMod L_N A$ are in bijection with
the $2$-extensions in $\Mod A$. For such an $S_i$ we can find a projective resolution as a topological module:
\[
\xymatrix{
\cdots\ar@{^(->}[r]&\bigoplus_{t(r)=i}P_{h(r)}\ar@{^(->}[r]^{\cdot ra^{-1}}&\bigoplus_{t(a)=i}P_{h(a)}\ar[r]^{\cdot a}&P_i\ar@{->>}[r]&S_i 
}
\]
the number of relations $r$ with $h(r)=i,t(r)=j \in I_N$ equals the dimension of the space $\Ext^2_{L_N A}(S_{i},S_{j})$, so when
 $i,j \in I_M$ we know this is equal to the dimension of $\Ext^2_A(S_{i},S_{j})\stackrel{CY}{=}\Hom(S_{j},S_{i})=\delta_{ij}$.
So if $i \in I_M$ we will denote by $r_i$ the unique relation with $h(r_i)=t(r_i)=i$.
For every $a$ with $h(a),t(a) \in I_M$ we define $f_a = r_{t(a)}a^{-1}$. Because  $r_i \in \ideal W_N^2$ we can split
$f_a = \sum_{b \in Q_{N1}}g_{ab}b + d_a$ where the first part is a linear combination of arrows and $d_a \in  \hat\cJ_N^2$.

We work out the composition of $\eta \in \Ext^1(S_i,S_j)$ and $\xi \in \Ext^1(S_j,S_i)$ with $i,j \in I_M$.
both extensions can be seen as maps: 
\[
\eta: \{a| h(a)=i, t(a)=j\} \to \C\text{ and }\xi: \{a| h(a)=j, t(a)=u\} \to \C
\]
We can use this $g$ to explicitly calculate the pairing.
\[
\Ext^1(S_j,S_i) \times \Ext^1(S_i,S_j) \to \Ext^2(S_i,S_i): (\xi_a)*(\eta_b) = \sum_{ab}g_{ab}\xi_a\eta_b.
\]
Property $C3$ for Calabi Yau algebras now implies that there are scalars $\alpha_i, i \in Q_0$ (coming from the traces) such that 
$\alpha_{t(a)}g_{ab}$ is antisymmetric and non-degenerate.
Using a base transformation on the arrows we can put $\alpha_{t(a)}g_{ab}$ in its standard symplectic form.
The fact that $g_{ab}\ne 0 \implies h(a)=t(b) \wedge t(a)=h(b)$ indicates that this base transformation
only mixes arrows with identical head and tail. 
In this new basis the arrows can be partitioned in couples $(a,a^*)$ with $\alpha_{t(a)}g_{a^*a} = 1$ and
$\alpha_{t(a)}g_{ab}=0$ if $b\ne a^*$. The relation now becomes
$r_i = \sum_{t(a)=i,h(b)=i} \alpha_{t(a)}g_{ab}ba + d_i$ with $d_i \in \ideal W^3$.

If we now proceed to the tangent cone of $M$, we must factor out all $e_i$ with $i \not \in I_M$ and
then take the associated graded. 
This means that all relations that were not one of the $r_i$ become zero. The $r_i$ are in fact
the preprojective relation with some higher order terms, so if we can prove that the $\{r_i\}$ form a gradable
set (see \ref{gradedcomplete}) then we are done.
First assume that $Q_M$ is connected. We have to distinguish two cases. 
\begin{itemize}
\item If $Q_M$ is not the double of a Dynkin quiver then it was proved in \cite{eting} that the preprojective algebra on $Q$ has global dimension $2$.
This implies that there are no $\Ext$'s of degree $3$, and hence there are no relations between the preprojective relations which means
that the $\{r_i\}$ are gradable set.
\item If $Q_M$ is the double of a Dynkin quiver then the preprojective algebra is not of global dimension $2$ but using the work of \cite{schofield}, \cite{erdman}, \cite{king} we have a nice description of the syzygies. The preprojective algebra is finite dimensional and as it is graded we can look at the highest degree component that is nonzero. Call this degree $d$. For every vertex $v$ in the quiver there is a unique relation between the relations $r_{i\mm}$:
\[
 \sum_i a_{vi}r_{i\mm}b_{vi}=0\text{ with $\degree a_{vi}+\degree b_{vi}=d$}
\]
If we substitute $r_{i\mm}$ by $r_{i}$ and evaluate its homogeneous parts in the preprojective algebra we see that the must evaluate to $0$ because their
degree is to high. This implies that the $\{r_i\}$ form again a gradable set.
\end{itemize}
If $Q_M$ is not connected then we must look at each connected component separately because then the 
$C_MA$ and $L_MA$ are direct sums of the subalgebras supported by the connected components. 
\end{proof}

Take care: although most preprojective algebras are $2$-Calabi Yau themselves it is not true that $C_M A$ is
always a $2$-CY algebra. The easiest counterexample is the following: take $A$ to be the preprojective algebra
of the double of a quiver with one cycle, this is a Calabi Yau algebra by theorem \ref{calab}. Let $M$ be the semisimple module
that is the direct sum of all simple modules except one. The tangent cone is the preprojective algebra of the double
of a Dynkin quiver of type $A_n$. As we already noted, this algebra is not Calabi Yau.

If we want to consider the case of CY3, extra problems arise and at the moment there is no proof that works in all cases.
There are however some interesting partial results. First of all if the category $\TMod M$ is itself CY3 then Rouquier and Chuang \cite{rouquier} have proved that $L_MA$ must derive from a superpotential $W \in \C \hat Q_M$. 

If $\TMod M$ is not CY3 it is still sometimes possible to prove that $L_MA$ comes from a superpotential using a result by Ed Segal. In \cite{segal} he proved that if the $A^\infty$-structure on $\Ext_{A}(M,M)$ admits a bilinear structure $\<,\>$ for which it is cyclic i.e. 
$\<a_0,m_k(a_1,\dots,a_k)\> = \<a_1,m_k(a_2,\dots,a_0)\>$ then $L_MA$ derives from a superpotential. It is not clear whether 
this always holds if $A$ is CY3, but it holds if $A$ is itself derived from a superpotential. So if $A$ is a graded CY3 algebra and $M$ any semisimple representation then $L_MA$ is derived from a superpotential. 

The fact that $L_MA$ derives from a superpotential does not necessary imply that $C_MA$ derives from a superpotential.
This is true if $W_\mm$ is a nondegenerate superpotential (i.e. $\C Q/(\partial_a W_{\mm})$ is CY3) but in the other cases
extra relations between the derivative $\partial W_{\mm}$ might prevent $\{\partial_a W\}$ might from being a gradable set.
At the moment there is not enough known for algebras with degenerate superpotentials to rule out these problems.

\section{Examples}

In this final section we will describe two examples coming from fundamental groups of compact orientable aspherical manifolds (i.e. for which the higher dimensional homotopy groups vanish). It was pointed out by Kontsevich and Ginzburg \cite{ginzburg} that the group algebras of such groups are CYd where $d$ is the real dimension of the manifold. These algebras are not graded and do in general not derive from preprojective relations or superpotentials.
Hence they form a nice class of algebras for which we can try to compute the tangent cone.
\subsection{Compact surfaces}
The fundamental group of a compact orientable surface is a group of the form
\[
G := \<X_1,Y_1,\dots,X_g,X_g| X_1Y_1X_1^{-1}Y_1^{-1}\cdots X_gY_gX_g^{-1}Y_g^{-1}\>
\]
where $g$ is the genus of the surface.

Let us look at some tangent cones for some representations of orientable compact surface groups.
If $W$ is a one-dimensional representation one can manually check that 
\[
 C_W \C G \cong \C\<X_1,Y_1,\dots,X_g,Y_g\>/\<[X_1,Y_1]+\cdots+[X_g,Y_g]\> =: \Pi_{2g}
\]
Which is indeed the preprojective algebra on one vertex and $2g$ loops.

In \cite{Rapinchuk} it is shown that for every $n$ $\Rep_n \C G$ is an irreducible variety.
We can compute its dimension using the results on preprojective algebras from \cite{momentmap}.
\[
 \Dim \Rep_n \Pi_g = \begin{cases}
                                                                           n^2+n &g=1\\
(2g-1)n^2+1&g>1
                                                                          \end{cases}
\]
Now if $g>1$ then every variety $\Rep_n \Pi_g$ contains simples and hence so does
$\Rep_n \C G$. The tangent cone of such a simple representation $W$ must again be of the
form $\Pi_h$ for some $h \in \N$. We must have that
\se{
 \Dim \Rep_n \Pi_g &= \Dim \Rep_1 \Pi_h + \Dim \cO_W\\
 (2g-1)n^2 +1 &= 2h + n^2 -1\\
 (g-1)n^2 + 1 &= h
}
So $h$ only depends on the dimension of the simple representation.
In a similar way we can deduce that if $W_1$ and $W_2$ are 
simple representations with dimension $n_1$ and $n_2$, then the number of arrows between
the two simple components will be $a \in \N$ such that  
\se{
 \Dim \Rep_{n_1+n_2} \Pi_g &= \Dim \Rep_{(1,1)}C_{{W_1\oplus W_2}}\C G  + \Dim \cO_{W_1\oplus W_2}\\
 (2g-1)(n_1+n_2)^2 +1 &= 2(h_1+h_2+a)-1 + (n_1+n_2)^2-2\\
(2g-2)(n_1+n_2)^2 +4 &= 2((g-1)n_1^2 + (g-1)n_2^2 + 2 +a)\\
(g-1)((n_1+n_2)^2  -n_1^2 - n_2^2)
&=a\\
2n_1n_2(g-1)&=a.
}
This implies that the local quiver and the tangent cone of any semisimple representation
can be determined by the dimensions of the simples:
\begin{theorem}
If $M =S_1^{\oplus \eps_1} \oplus  \cdots \oplus S_k^{\oplus \eps_k}$ is a semisimple $\C \fg_g$ representation and
the factor $S_i$ has dimension $n$, then the local quiver of $M$ has $k$ vertices with on the $i^{th}$ vertex $2(g-1)n_i^2+2$ loops
and $2n_in_j(g-1)$ arrows from the $i^{th}$ to the $j^{th}$ vertex.
\end{theorem}

In general if $A$ is a 2-CY algebra we can deduce some information about the representation spaces using the tangent cones.
First of all we know that every connected component of $\Rep_n A$ must be irreducible because this is the case for the preprojective
algebras. This allows us to define the component semigroup $\CS(A)$. Its elements are the connected components of all $\Rep_n A$.
The group operation is induced by the direct sum:
\[
\gamma_1 + \gamma_2 = \gamma_3 \iff \exists W_1 \in \gamma_1:\exists W_2 \in \gamma_2: W_1 \oplus W_2 \in \gamma_3.
\]
Two simples in the same component $\gamma$ have the same local quiver because the number of loops in the local quiver is
equal to the dimension of $\gamma/\!\!/\GL_n$. A simple is never located at a singularity because $\Rep_1 \Pi(Q)$, where $Q$
is a quiver with one vertex, is always smooth.
 
If we have one simple $T$ in $\gamma_1$ and two simples $S_1,S_2$ in $\gamma_2$ then the local quiver and hence the tangent cone of
$T\oplus S_1$ and $T\oplus S_2$ are the same. This is because if the double quiver is of the form
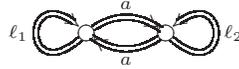
\[
\xymatrix{\vtx{}\ar@2@(ld,lu)^{\ell_1}\ar@2@/^/[r]^{a}&\vtx{}\ar@2@(rd,ru)_{\ell_2}\ar@2@/^/[l]^{a}} 
\]
We can deduce $\ell_i$ from the dimensions of $\gamma_i/\!\!/\GL_{n_i}$ and $a$ from the dimension of $\gamma_1+\gamma_2/\!\!/\GL_{n_1+n_2}$:
\[
\dim \gamma_1+\gamma_2/\!\!/\GL_{n_1+n_2} = \ell_1+\ell_2 +2a -2.
\]
If we have a finite set of generators for the component semigroup $\gamma_1,\dots,\gamma_u$ and we chose a simple representation $S_i$
in each component $\gamma_i$, we can construct the tangent cone of $M = S_1 \oplus \dots \oplus S_n$: $\Glob A:=C_M A$ this
tangent cone can be seen as a global model of $A$ in the sense that
for every semisimple $A$-module $M$ there exists a $\Glob A$-module $M'$ with $C_M A\cong C_{M'}A$. This notion of global quiver is a generalization of
the one quiver in \cite{onequiver}.

\subsection{The Heisenberg Algebra}

The integral Heisenberg algebra is the group
\[
 \grp{H} := \left\{ \left[ \begin{smallmatrix} 1&a&b\\0&1&c\\ 0&0&1\end{smallmatrix} \right]: a,b,c \in \Z\right\}
\]
it can be presented as a group generated by two elements which both commute with their commutator.
\[
\< X,Y| \{X,\{X,Y\}\} ,\{Y,\{X,Y\}\}\> \text{with $\{X,Y\}=XYX^{-1}Y^{-1}$}.
\]
The corresponding group algebra has the following presentation
\[
\C\<X,Y,X^{-1},Y^{-1}\>/\<XYX^{-1}Y^{-1}-YX^{-1}Y^{-1}X, XYX^{-1}Y^{-1}-Y^{-1}XYX^{-1}\>
\]
It is a CY3 algebra because the Heisenberg group is the fundamental group of the quotient 
\[
 \left\{ \left[ \begin{smallmatrix} 1&a&b\\0&1&c\\ 0&0&1\end{smallmatrix} \right]: a,b,c \in \R\right\}/\grp{H}.
\]
The left side of the expression is contractible so the quotient is aspherical.

To calculate the tangent cones we need some technical machinery.
Suppose $M$ is a simple representation of $A =\C Q/\cI$. Define $B_\eps$ to be the open disc with radius $\eps$ around the origin in $\C$
and let $\theta: B_\eps^k \to \Rep_\alpha Q$ be an analytical parameterization 
such that 
\begin{itemize}
 \item $\theta(0)=M$
 \item $T_M\Image \theta \cap T_M \GL_{\alpha} M=0$
 \item $\GL_\alpha \Image \theta$ covers an analytical open neighborhood of $M$ in $\Rep_{\alpha} A$ (so not necessarily the whole $\Rep_{\alpha} Q$).
\end{itemize}
This implies that we can expand $\theta$ as
\[
 \theta(t) = \rho_M +  t_i\theta_i+ t_it_j\theta_{ij} +\dots
\]
For every $\mu \in \N$ we can define an expanded version of $\theta$ that acts on $\mu\times\mu$-matrices $T_i$ instead of scalars
and has values in $\Rep_{\mu\alpha} Q$. 
\[
\theta^{(\mu)}(T) := \id_{\mu}\otimes \rho_M +  T_i\otimes \theta_i+ T_iT_j\otimes\theta_{ij} +\dots
\]
Here $T$ is a $k$-vector of $\mu\times \mu$-matrices. This map is by construction $\GL_\mu$-covariant. 

If $r$ is a relation in $\cI$ then we can calculate $r(\theta^{(\mu)})$ which gives us for every $\alpha$-entry of $r$
a $\GL_\mu$-covariant function $r_{\kappa}: \Mat_\mu(\C)^k \to \Mat_\mu(\C)$. We can split every $r_{\kappa}$ in its homogeneous parts $r_{\kappa i}$.
These functions are compatible with taking the direct sum so using the techniques from the reconstruction theorem \ref{recon}
they can be expressed as a noncommutative monomials in the variables $T$ and the $r_\kappa$ as noncommutative power series.

\begin{theorem}\label{param}
$L_M A := \C\<\!\!\<T_1,\dots, T_k\>\!\!\>/\<r_{\kappa}\>$
\end{theorem}
\begin{proof}
The embedding $\theta^{(\mu)}$ is in fact a slice of $\Rep_{\mu\alpha} Q$ at $M^\mu$. This slice is however not an algebraic slice but an analytical one.
The restriction to the zero locus of the $r_{\kappa}$ gives a slice of $\Rep_{\mu\alpha} A$.
We can in fact replicate the reasoning in the reconstruction theorem \ref{recon}, adapted to complete instead of graded algebras and use analytical slices.
\end{proof}

We will now apply this theorem to the case of simple representations of the integral Heisenberg group.
\begin{theorem}
The tangent cone of any simple module of the integral Heisenberg group algebra is isomorphic to
the algebra coming from the superpotential $X^2Y^2-XYXY$.
\[
 C_W \C \grp H \cong \C\<X,Y\>/\<\partial_X(X^2Y^2-XYXY),(\partial_Y(X^2Y^2-XYXY)\>.
\]
\end{theorem}
\begin{proof}
 In \cite{simplesheisenberg} the simple $n$-dimensional representations of this group have been classified and they can be written in the following form
\[
\rho_{a,b}(X) := \begin{pmatrix}
       0&a&0&\dots&0\\
       0&0&a&\dots&0\\
       \vdots&&&\ddots&\dots\\
       a&0&0&\dots&0\\
      \end{pmatrix}
\text{ and }
\rho_{a,b}(Y) := \begin{pmatrix}
       be^{2\pi i/n}&0&\dots&0\\
       0&be^{4\pi i/n}&\dots&0\\
       \vdots&&\ddots&\dots\\
       0&\dots&0&be^{2n\pi i/n}\\
      \end{pmatrix}
\]
So for a given $\rho_{a,b}$ we have a $2$-parameter family $\theta: (t_1,t_2) \to \rho_{a(1+t_1),b(1+t_2)}$ that satisfies the conditions
of theorem \ref{param}.
We get the following equalities
\se{
 \theta(T)(X) &= 1 \otimes \rho_{a,b}(X) + T_1 \otimes \rho_{a,b}(X)\\
 \theta(T)(X^{-1}) &= 1 \otimes \rho_{a,b}(X) - T_1 \otimes \rho_{a,b}(X)+  T_1^2 \otimes \rho_{a,b}(X)+\dots\\
 \theta(T)(Y) &= 1 \otimes \rho_{a,b}(Y) + T_2 \otimes \rho_{a,b}(Y)\\
 \theta(T)(Y^{-1}) &= 1 \otimes \rho_{a,b}(Y) - T_2 \otimes \rho_{a,b}(Y)+  T_2^2 \otimes \rho_{a,b}(Y)+\dots\\
}
so the relation $r= XYX^{-1}Y^{-1} - YX^{-1}Y^{-1}X$ becomes
\se{
r(\theta(T)) &=(1+T_1)(1+T_2)(1-T_1+T_1^2+\dots)(1-T_2+T_2^2+\dots) \otimes \rho_{a,b}(XYX^{-1}Y^{-1})\\
&~-(1+T_2)(1-T_1+T_1^2+\dots)(1-T_2+T_2^2+\dots)(1+T_1) \otimes \rho_{a,b}(YX^{-1}Y^{-1}X)\\
&= ((1+T_1)(1+T_2)(1-T_1+T_1^2+\dots)(1-T_2+T_2^2+\dots)\\
&~-(1+T_2)(1-T_1+T_1^2+\dots)(1-T_2+T_2^2+\dots)(1+T_1)) \otimes 1\\
&= (T_1^2T_2 +T_2T1^2 - 2T_1T_2T_1)\otimes 1+\cdots\\
&= \partial_{T_1}(T_1^2T_2^2 - T_1T_2T_1T_2)\otimes 1+\cdots 
}
The second equation transforms analogously. Now it is known that $\C\<X,Y\>/\<\partial_X(X^2Y^2-XYXY),(\partial_Y(X^2Y^2-XYXY)\>$
is CY3 \cite{AS},\cite{raf} and hence we can take the minimal parts of the relations to get the tangent cone.
\end{proof}

One can also calculate that there are no extensions between different simples this implies that all tangent cones
are of the form 
\[
 C_M A := [\C\<X,Y\>/\<\partial_X(X^2Y^2-XYXY),(\partial_Y(X^2Y^2-XYXY)\>]^{\oplus k}
\]
for $M= S_1^{\oplus \eps_1}\oplus \cdots \oplus S_k^{\oplus \eps_k}$.

\bibliographystyle{amsplain}
\bibliography{local}
\end{document}